\documentclass[11pt,letterpaper,reqno]{amsart}

\usepackage{tikz}
\usetikzlibrary{positioning,shapes.geometric,arrows.meta}
\usepackage{amssymb}
\usepackage{amsmath}
\usepackage{amsthm}
\usepackage{amsfonts}
\usepackage{bbm}
\usepackage{enumitem}
\usepackage{pgfplots}
\pgfplotsset{compat=1.18}
\usepackage{booktabs}
\usepackage{graphicx}
\usepackage[T1]{fontenc}
\usepackage{doi}
\usepackage{float}
\usepackage[expansion=false]{microtype}
\usepackage{tabularx}
\usepackage{array}
\usepackage{xcolor}
\usepackage[numbers]{natbib}

\usepackage{hyperref}
\usepackage{bookmark}
\hypersetup{
  pdfstartview={FitH},
  colorlinks=true,
  linkcolor=blue,
  citecolor=blue,
  urlcolor=blue
}

\newtheorem{thm}{Theorem}[section]
\newtheorem{lem}[thm]{Lemma}
\newtheorem{prop}[thm]{Proposition}

\newtheorem{ques}[thm]{Question}

\theoremstyle{definition}

\newtheorem{rem}[thm]{Remark}

\numberwithin{equation}{section}

\newcommand{\CC}{\mathbb{C}}
\newcommand{\HH}{\mathbb{H}}
\newcommand{\RR}{\mathbb{R}}
\newcommand{\ZZ}{\mathbb{Z}}
\newcommand{\wideC}{\widehat{\CC}}
\newcommand{\dd}{\,\mathrm{d}}
\newcommand{\PSL}{\operatorname{PSL}}

\begin{document}

\title[A half-plane counterexample]
{A Counterexample to Nevanlinna's Century-Old Half-Plane Problem}

\author[Y.~He]{Yixin He}
\address{School of Mathematical Sciences, Fudan University, Shanghai 200433, P.~R.~China}
\email{yixin.he717@gmail.com}

\author[T.~Zhang]{Teng Zhang}

\address{School of Mathematics and Statistics, Xi'an Jiaotong University, Xi'an 710049, P.~R.~China}
\email{teng.zhang@stu.xjtu.edu.cn}



\begin{abstract} Let $\HH=\{z\in\CC:\operatorname{Im}z>0\}$, and let $N(\HH)$ denote
	the Nevanlinna class in $\HH$. 
We construct a nonconstant meromorphic function $F$ on $\CC$ such that $F^{-1}(\{0,1,\infty\})\subset\RR$ and $F|_{\HH}\notin N(\HH)$. Thus, omitting three distinct spherical values in a half-plane does not force a meromorphic function on $\CC$ to be of bounded type there. This provides a counterexample to Nevanlinna's century-old half-plane problem.
\end{abstract}

\maketitle

\section{Introduction}

Let $\HH=\{z\in\CC:\operatorname{Im}z>0\}$.  The Nevanlinna class $N(\HH)$ consists of the meromorphic functions in $\HH$ representable as quotients of two bounded analytic functions.  Equivalently, for an analytic function $g$, membership in $N(\HH)$ is characterized by the existence of a harmonic majorant of $\log^+|g|$; see \cite[Chapter~II, Section~5, pp.~66--71]{Gar07}.  

In this paper, we consider the following long-standing question, which we refer to as \emph{Nevanlinna's Century-Old Half-Plane Problem}. It goes back to Nevanlinna's 1925 work \cite{Nev25}; see also \cite[Question~2]{EKS26}:
\begin{ques}[Nevanlinna]\label{ques:nevanlinna}
Suppose that a meromorphic function on $\CC$ omits three distinct values of $\wideC$ in a half-plane.  Must its restriction to that half-plane belong to the Nevanlinna class?
\end{ques}

Question~\ref{ques:nevanlinna} has a long history in value
distribution theory and the theory of meromorphic functions, which may be
summarized chronologically as follows:

\begin{itemize}
	
	\item
	In 1925, Nevanlinna \cite{Nev25} proved that if $F$ is meromorphic in
	$\CC$, has finite order, and omits three distinct values of $\wideC$ in a
	half-plane, then its restriction to that half-plane is of bounded type.
	Thus, Question~\ref{ques:nevanlinna} has an affirmative answer under the
	additional assumption that $F$ has finite order.
	
\item
In 1955, Edrei \cite{Edr55} proved, as a
special case of his theorem on radially distributed values, that if all
zeros, one-points, and poles of a meromorphic function lie on the real
axis and $
\delta(0,F)+\delta(1,F)+\delta(\infty,F)>0,$
then the order of $F$ is at most $1$.  Combining this result with
Nevanlinna's finite-order theorem
\cite{Nev25}, one obtains
that $F$ is of bounded type in both the upper and lower half-planes.
Thus Edrei's theorem gives an affirmative answer in this special case
only under the additional positive-deficiency hypothesis.
	
	\item
	In 1961, Ostrovskii \cite{Ost61} substantially weakened the finite-order
	assumption. He showed that the same conclusion holds provided
	$\int_{1}^{\infty}\log^{+}T(r,F)\,r^{-2}\,dr<\infty$, where $T(r,F)$
	denotes the Nevanlinna characteristic of $F$. Since every finite-order
	meromorphic function satisfies this condition, Ostrovskii's theorem
	strictly extends Nevanlinna's result.
	
\item
In 1975, Gol'dberg \cite{Gol75} showed that, for general meromorphic
functions in a half-plane, an analogue of Nevanlinna's lemma on the
logarithmic derivative fails. More precisely, he constructed examples
showing that the logarithmic-derivative estimates used in the classical
approach cannot hold without additional assumptions. This obstructs a
direct extension of the method used by Nevanlinna and Ostrovskii to the
unrestricted case.
	
\item
In 2026, Eremenko, Kulikov, and Sodin \cite{EKS26} observed that for
functions defined only in a half-plane, the corresponding statement is
false: the modular function $\lambda$ is meromorphic in $\mathbb H$, omits
the three values $0$, $1$, and $\infty$, but is not of bounded type there.
They further studied a quantitative version of the problem for curved
half-planes $\mathbb H(-m)=\{x+iy:y>-m(x)\}$, where $m$ is even,
positive, continuous, and nonincreasing on $[0,\infty)$.  They proved
\cite[Theorem~1]{EKS26} that if $F$ is
meromorphic in $\mathbb H(-m)$, omits three distinct values, and
$\int_{1}^{\infty}\log^{-}m(t)\,t^{-2}\,dt<\infty$, then
$F|_{\mathbb H}$ is of bounded type. Moreover, this condition is sharp:
if $\int_{1}^{\infty}\log^{-}m(t)\,t^{-2}\,dt=\infty$, then there exists
a meromorphic function in $\mathbb H(-m)$ omitting three distinct values
whose restriction to $\mathbb H$ is not of bounded type.

\end{itemize}

%Recently, Eremenko, Kulikov, and Sodin \cite[Theorem~1]{EKS26} established a sharp logarithmic-integral dichotomy for meromorphic functions omitting three values in certain variable enlargements of the upper half-plane.

The purpose of this paper is to give a negative answer to Question~\ref{ques:nevanlinna}.

\begin{thm}\label{thm:main}
There exists a nonconstant meromorphic function $F\colon\CC\to\wideC$ such that $F(z)\notin\{0,1,\infty\}$ for every $z\in\HH$, but $F|_{\HH}\notin N(\HH)$. In fact, $F^{-1}\{0,1,\infty\}\subset\RR$.
\end{thm}

We now give a brief sketch of our construction in Theorem~\ref{thm:main}.

\medskip 
\noindent \textbf{Sketch of our construction.} The construction begins with the modular lambda function
\[
\lambda\colon\HH\longrightarrow\wideC\setminus\{0,1,\infty\}.
\]
Although $\lambda$ is not of bounded type in $\HH$, it cannot itself serve as the required example, because it has no meromorphic continuation through the whole real axis.  We instead construct a simply connected domain $\Omega\subset\HH$ whose boundary is a locally finite chain of Farey edges.  If $W\colon\HH\to\Omega$ is a suitably normalized Riemann map, then $\lambda\circ W$ has finite real boundary values away from a locally finite set.  At the exceptional points, finite-width modular cusps become zeros, one-points, or poles after reflection.  Consequently, $\lambda\circ W$ extends meromorphically to $\CC$.

The domain $\Omega$ is selected by a diagonal procedure.  In pairwise disjoint bays, we make finite Stern--Brocot refinements of individually chosen finite depths.  The limiting infinite refinement of one bay reaches a real interval.  Near that interval its Green function is bounded below by a positive multiple of the height, whereas the height-weighted spherical area of $\lambda$ is infinite.  Monotone convergence therefore permits a finite refinement whose local Green energy exceeds one.  The final domain contains each selected one-bay domain, so these disjoint contributions add to infinity.

%Throughout the paper, the Green function is normalized by
%\[
%G_D(p,z)>0,
%\qquad
%-\Delta_z G_D(p,z)=2\pi\delta_p,
%\]
%and the spherical derivative of a meromorphic function $g$ is
%\[
%g^{\#}(z)=\frac{|g'(z)|}{1+|g(z)|^2}.
%\]

\medskip 
\noindent \textbf{Notation.}
Throughout the paper, if $D\subset\CC$ is a Greenian domain and $p\in D$, we denote by $G_D(p,z)$ the Green function of $D$ with pole at $p$, normalized so that $G_D(p,z)>0$ for $z\in D\setminus\{p\}$ and
\[
-\Delta_z G_D(p,z)=2\pi\delta_p
\]
in the sense of distributions, with zero boundary values in the usual Green-function sense. Here $\Delta=\partial_x^2+\partial_y^2$ is the Euclidean Laplacian and $\delta_p$ denotes the Dirac point mass at $p$. Equivalently,
\[
G_D(p,z)=\log\frac{1}{|z-p|}+O(1),\qquad z\to p.
\]

For a meromorphic function $g$, its spherical derivative is defined by
\[
g^{\#}(z)=\frac{|g'(z)|}{1+|g(z)|^2}
\]
at points where $g(z)\neq\infty$, and by continuous extension at the poles of $g$. Equivalently, near a pole one may use the identity $g^{\#}=(1/g)^{\#}$.

\medskip
\noindent\textbf{Organization of the paper.}
In Section~\ref{sec:green-energy}, we establish the Green-energy obstruction
to membership in the Nevanlinna class.
In Section~\ref{sec:lambda}, we collect the required properties of the
modular lambda function and prove the logarithmic-mean estimate and the
resulting divergence of the height-weighted spherical area.
In Section~\ref{sec:farey}, we introduce the finite Farey refinements used
in the boundary construction.
In Section~\ref{sec:diagonal}, we carry out the Green-kernel diagonal
argument and construct a domain with infinite Green-weighted spherical area.
In Section~\ref{sec:reflection}, we prove that, after uniformization,
Schwarz reflection across the Farey boundary produces a meromorphic function
on the whole plane.
In Section~\ref{sec:main-proof}, we combine these ingredients to prove
Theorem~\ref{thm:main}.
Finally, in Section~\ref{sec:further-remarks}, we explain how the construction
fits with the classical results of Nevanlinna and Edrei and record several
consequences.

\medskip
\noindent\textbf{Acknowledgements.}  We thank Professors Eremenko, Kulikov, and Sodin for their helpful comments on this work. The second author is also  deeply grateful
to his advisor Professor Zongben Xu for his kind support and concern regarding both his academic work
and personal life. Teng Zhang is supported by the China Scholarship Council, the Young Elite Scientists Sponsorship Program for PhD Students of the China Association for Science and Technology, and the Fundamental Research Funds for the Central Universities at Xi'an Jiaotong University (Grant No.~xzy022024045). This counterexample was obtained through AI-assisted exploration under the authors' mathematical supervision and guidance. The authors take full responsibility for the mathematical arguments and the correctness of the final manuscript.

\section{Green energy and the Nevanlinna class}\label{sec:green-energy}

We first isolate the one direction of the Ahlfors--Shimizu criterion needed later.  We use the Riesz decomposition theorem in the normalization of \cite[p.~76, Theorem~3.7.9]{Ran95}.

\begin{prop}\label{prop:energy-obstruction}
Let $D$ be a Greenian plane domain, let $p\in D$, and let $g$ be holomorphic in $D$.  If $\log^+|g|$ has a harmonic majorant in $D$, then
\[
\int_D G_D(p,z)\bigl(g^{\#}(z)\bigr)^2\dd A(z)<\infty.
\]
Consequently, if the integral is infinite for $D=\HH$, then $g\notin N(\HH)$.
\end{prop}

\begin{proof}
Set
\[
u(z)=\frac12\log\bigl(1+|g(z)|^2\bigr).
\]
Then
\[
0\leq u\leq\log^+|g|+\tfrac12\log 2,
\qquad
\Delta u=2\bigl(g^{\#}\bigr)^2.
\]
Thus $u$ has a finite harmonic majorant $H$.  Let $D_j\Subset D$ be a smooth exhaustion with $p\in D_1$.  The nonnegative function $H-u$ is superharmonic, and the Riesz decomposition theorem on $D_j$ gives
\[
\frac{1}{2\pi}\int_{D_j}G_{D_j}(p,z)\Delta u(z)\dd A(z)
   \leq H(p)-u(p).
\]
As $j\to\infty$, the Green kernels increase to $G_D$; hence monotone convergence yields
\[
\frac1\pi\int_D G_D(p,z)\bigl(g^{\#}(z)\bigr)^2\dd A(z)
   \leq H(p)-u(p)<\infty.
\]

%For the last assertion, suppose that $g=h_1/h_2$ with bounded analytic $h_1,h_2$ in $\HH$.  Since $g$ is holomorphic, the zeros of $h_2$ occur among those of $h_1$, with at least the same multiplicities.  Dividing both functions by the corresponding Blaschke product leaves a bounded zero-free denominator.  After a harmless common scaling, minus the logarithm of its modulus is a positive harmonic majorant of $\log^+|g|$.  This is also an immediate consequence of the canonical factorization in \cite[p.~71, Theorem~5.5]{Gar07}.  The first part of the proposition now proves the contrapositive.

For the last assertion, suppose that
$g=h_1/h_2$, where $h_1,h_2\in H^\infty(\HH)$ and
$h_2\not\equiv0$.  Let $B$ be the Blaschke product formed from the zeros
of $h_2$, counted with multiplicity.  Since $g$ is holomorphic, every zero
of $h_2$ is a zero of $h_1$ with at least the same multiplicity.  Hence
\[
\widetilde h_j=\frac{h_j}{B}\in H^\infty(\HH),
\qquad j=1,2,
\]
and $\widetilde h_2$ is zero-free.  After a common scaling, we may assume
that $
\|\widetilde h_1\|_\infty,
\|\widetilde h_2\|_\infty\leq1.$
It follows that
\[
\log^+|g|
=
\log^+\left|\frac{\widetilde h_1}{\widetilde h_2}\right|
\leq-\log|\widetilde h_2|.
\]
Since $\widetilde h_2$ is zero-free, the function
$-\log|\widetilde h_2|$ is a positive harmonic majorant of
$\log^+|g|$.  This also follows from the canonical factorization in
\cite[p.~71, Theorem~5.5]{Gar07}.  The first part of the proposition now
proves the contrapositive.
\end{proof}

We shall repeatedly use two standard properties of Green functions: domain monotonicity and convergence under exhaustion; see \cite[p.~108, Corollary~4.4.5 and Theorem~4.4.6]{Ran95}. If $D_1\subset D_2$ are Greenian domains and $p\in D_1$, then
\[
G_{D_1}(p,z)\leq G_{D_2}(p,z),
\qquad z\in D_1.
\]
Moreover, if $(D_m)$ is an increasing sequence of domains with $D_m\uparrow D$, where $D$ is Greenian and $p\in D_1$, then, for every $z\in D\setminus\{p\}$,
\begin{equation}\label{eq:green-exhaustion}
	G_{D_m}(p,z)\uparrow G_D(p,z)
\end{equation}
as $m\to\infty$, once $m$ is large enough that $z\in D_m$.

Whenever Green functions on varying subdomains are integrated over a fixed ambient domain, we extend them by zero outside their respective domains. With this convention, $G_{D_m}(p,\cdot)$ is pointwise increasing on $D\setminus\{p\}$ and converges there to $G_D(p,\cdot)$.

\section{The modular lambda function}\label{sec:lambda}

Put $q=e^{\pi i\tau}$ and define
\[
\vartheta_2(\tau)=\sum_{n\in\ZZ}e^{\pi i(n+1/2)^2\tau},
\qquad
\vartheta_3(\tau)=\sum_{n\in\ZZ}e^{\pi in^2\tau},
\qquad
\lambda(\tau)=\frac{\vartheta_2(\tau)^4}{\vartheta_3(\tau)^4}.
\]
This fixes our normalization.  The following standard facts are recorded in \cite[Chapter~7, Sections~3.4--3.5, pp.~278--282]{Ahl79} and, in the theta-function normalization used above, in \cite[Equations~23.15.6, 23.17.4, 23.18.1, and 23.18.3]{NIS26}:
\begin{enumerate}[label=\textup{(\roman*)},leftmargin=2.4em]
\item $\lambda$ is holomorphic in $\HH$, is $2$-periodic, and omits $0$, $1$, and $\infty$ there;
\item $\lambda$ is invariant under the principal congruence group $\Gamma(2)$;
\item for every $\gamma\in\PSL_2(\ZZ)$, there is an anharmonic transformation $R_\gamma$ satisfying $\lambda\circ\gamma=R_\gamma\circ\lambda$, where
\begin{equation}\label{eq:anharmonic}
R_\gamma\in\mathcal A:=
\left\{
t\mapsto t,\ t\mapsto\frac1t,\ t\mapsto1-t,\
t\mapsto\frac1{1-t},\ t\mapsto\frac{t}{t-1},\
t\mapsto\frac{t-1}{t}
\right\};
\end{equation}
\item uniformly modulo the period $2$,
\begin{equation}\label{eq:lambda-cusp}
\lambda(\tau)=16q+O(q^2),\qquad \operatorname{Im}\tau\to+\infty;
\end{equation}
\item $\lambda(iy)$ is real for $y>0$.
\end{enumerate}

The quantitative input for the construction is a lower bound for the logarithmic mean of $\lambda$.  We include the orbit-counting argument, following the idea of \cite[Lemma~6]{EKS26}, in order to make every normalization and multiplicity explicit.

\begin{lem}\label{lem:log-mean}
There exist constants $c>0$ and $y_0>0$ such that
\[
\int_0^2\log^+|\lambda(x+iy)|\dd x
   \geq c\log\frac1y,
\qquad 0<y<y_0.
\]
\end{lem}

\begin{proof}
Periodicity and \eqref{eq:lambda-cusp} show that
\[
\Lambda(q):=\lambda(\tau),\qquad q=e^{\pi i\tau},
\]
is well defined and holomorphic in the unit disk, with $\Lambda(0)=0$.  Let $r=e^{-\pi y}$.  For $a\in\CC\setminus\{0\}$, the elementary inequality
\[
\log^+|u-a|\leq\log^+|u|+C_a
\]
and Jensen's formula applied to $\Lambda-a$ give
\[
\int_{-1}^{1}\log^+|\lambda(x+iy)|\dd x
 \geq 2\sum_{|q_k|\leq r}\log\frac{r}{|q_k|}-C'_a,
\]
where the $q_k$ are the zeros of $\Lambda-a$, counted with multiplicity.  If $|q_k|=r$ for some $k$, perturb $r$ and pass to the limit; such a boundary zero makes zero contribution.  These zeros correspond to the $a$-points $\tau_k$ of $\lambda$ in the period strip $-1\leq\operatorname{Re}\tau<1$, and $|q_k|=e^{-\pi\operatorname{Im}\tau_k}$.  If $\operatorname{Im}\tau_k\geq2y$, then
\[
2\log\frac r{|q_k|}
 =2\pi(\operatorname{Im}\tau_k-y)
 \geq\pi\operatorname{Im}\tau_k.
\]
Consequently,
\begin{equation}\label{eq:jensen-orbit}
\int_{-1}^{1}\log^+|\lambda(x+iy)|\dd x
 \geq
 \pi\!\sum_{\substack{-1\leq\operatorname{Re}\tau_k<1\\
                 \operatorname{Im}\tau_k\geq2y}}
 \operatorname{Im}\tau_k-C'_a.
\end{equation}

We next supply sufficiently many $a$-points.  For every primitive pair $(c,d)\in\ZZ^2$, choose $a_0,b_0\in\ZZ$ with $a_0d-b_0c=1$.  Replacing $(a_0,b_0)$ by $(a_0+nc,b_0+nd)$ for a suitable $n\in\ZZ$, we obtain
\[
M_{c,d}=\begin{pmatrix}a_0&b_0\\c&d\end{pmatrix}\in\operatorname{SL}_2(\ZZ)
\]
such that $z_{c,d}=M_{c,d}i$ satisfies
\begin{equation}\label{eq:orbit-coordinates}
-\frac12\leq\operatorname{Re}z_{c,d}<\frac12,
\qquad
\operatorname{Im}z_{c,d}=\frac1{c^2+d^2}.
\end{equation}
The stabilizer of $i$ in $\operatorname{SL}_2(\ZZ)$ is $\{\pm I,\pm S\}$, where $S=\bigl(\begin{smallmatrix}0&-1\\1&0\end{smallmatrix}\bigr)$.  Thus matrices sending $i$ to the same point have lower rows among
\[
(c,d),\quad(-c,-d),\quad(d,-c),\quad(-d,c),
\]
and each orbit point is counted at most four times.

For completeness, let
\[
L(R)=\#\{(c,d)\in\ZZ^2:0<c^2+d^2\leq R^2\}.
\]
Comparison with unit lattice squares gives $L(R)=\pi R^2+O(R)$.  M\"obius inversion therefore yields
\[
\begin{aligned}
N(R)
&:=\#\{(c,d)\in\ZZ^2:\gcd(c,d)=1,\ c^2+d^2\leq R^2\}\\
&=\sum_{k\leq R}\mu(k)L(R/k)
 =\frac6\pi R^2+O(R\log R).
\end{aligned}
\]
The constant $6/\pi^2$ is the classical density of coprime integer
pairs; see \cite[p.~354, Theorem~332]{HW08}.  The circular lattice-point
estimate, including the stated error term, follows directly from the
preceding lattice-square comparison and M\"obius inversion.
\begin{equation}\label{eq:primitive-reciprocal}
\sum_{\substack{\gcd(c,d)=1\\0<c^2+d^2\leq X}}
\frac1{c^2+d^2}\geq c_1\log X
\end{equation}
for all sufficiently large $X$.

Take $X=(2y)^{-1}$.  Equations \eqref{eq:orbit-coordinates} and \eqref{eq:primitive-reciprocal}, together with the fourfold multiplicity bound, imply
\[
\sum_{\substack{z\in\PSL_2(\ZZ)\cdot i\\
        -1<\operatorname{Re}z<1,\ \operatorname{Im}z\geq2y}}
\operatorname{Im}z
 \geq c_2\log\frac1y.
\]
Let $
\mathcal S=\{R(\lambda(i)):R\in\mathcal A\}.$
This is a fixed finite subset of $\CC\setminus\{0,1\}$.  By \eqref{eq:anharmonic}, every point of the orbit $\PSL_2(\ZZ)\cdot i$ is an $a$-point of $\lambda$ for some $a\in\mathcal S$.  Partitioning the last sum according to this value, we may choose $a=a_y\in\mathcal S$ whose $a$-points contribute at least $|\mathcal S|^{-1}$ of the total.  The constants $C'_a$ in \eqref{eq:jensen-orbit} are uniformly bounded because $\mathcal S$ is finite.  After decreasing $y_0$ and choosing a smaller constant $c>0$, \eqref{eq:jensen-orbit} gives the asserted estimate on $[-1,1]$, and hence on $[0,2]$ by periodicity.
\end{proof}

We now turn the logarithmic mean into the weighted spherical-area divergence used in the diagonal construction.

\begin{lem}\label{lem:height-area}
For every $a\in\RR$ and every $\varepsilon>0$,
\begin{equation}\label{eq:height-area}
\int_0^\varepsilon\int_a^{a+2}
y\bigl(\lambda^{\#}(x+iy)\bigr)^2\dd x\dd y=\infty.
\end{equation}
\end{lem}

\begin{proof}
It suffices to consider small $\varepsilon$.  Put
\[
v(\tau)=\frac12\log\bigl(1+|\lambda(\tau)|^2\bigr),
\qquad
M(y)=\int_a^{a+2}v(x+iy)\dd x.
\]
Since $\lambda$ is $2$-periodic and $\Delta v=2(\lambda^{\#})^2$, integration of the $x$-derivative over one period gives
\[
M''(y)=2\int_a^{a+2}\bigl(\lambda^{\#}(x+iy)\bigr)^2\dd x.
\]
Since $v(\tau)\geq \log^+|\lambda(\tau)|$,
Lemma~\ref{lem:log-mean} and periodicity imply
\[
M(y)\geq c\log\frac1y
\]
for all sufficiently small $y>0$.  On the other hand, \eqref{eq:lambda-cusp} gives $M(y)\to0$ as $y\to\infty$.  The preceding identity makes $M$ convex, and a convex function with a finite limit at $+\infty$ has $M'\leq0$.  Hence, for $0<\delta<\varepsilon$,
\[
\begin{aligned}
2\int_\delta^\varepsilon\int_a^{a+2}
y\bigl(\lambda^{\#}(x+iy)\bigr)^2\dd x\dd y
&=\int_\delta^\varepsilon yM''(y)\dd y\\
&=\varepsilon M'(\varepsilon)-\delta M'(\delta)
   -M(\varepsilon)+M(\delta)\\
&\geq M(\delta)-C_\varepsilon,
\end{aligned}
\]
where $C_\varepsilon=M(\varepsilon)-\varepsilon M'(\varepsilon)$.  The right-hand side tends to $+\infty$ as $\delta\downarrow0$, proving \eqref{eq:height-area}.
\end{proof}

\section{Farey refinements}\label{sec:farey}

Write every rational number in reduced form.  Two rationals
\[
\alpha=\frac pq<\beta=\frac rs,
\qquad q,s>0,
\]
are Farey neighbors if $rq-ps=1$.  Let $e(\alpha,\beta)$ be the upper semicircle with diameter $[\alpha,\beta]$.  Its diameter and maximal height are
\begin{equation}\label{eq:farey-size}
\beta-\alpha=\frac1{qs},
\qquad
\operatorname{height}e(\alpha,\beta)=\frac1{2qs}.
\end{equation}
The mediant
\[
\alpha\oplus\beta=\frac{p+r}{q+s}
\]
is a Farey neighbor of both endpoints.  A single refinement replaces $e(\alpha,\beta)$ by
\[
e(\alpha,\alpha\oplus\beta)\cup e(\alpha\oplus\beta,\beta).
\]
Both children lie in the closed half-disk bounded by the parent semicircle and its diameter.  Indeed, on $[\alpha,\alpha\oplus\beta]$ the squared heights of the left child and the parent are, respectively,
\[
(x-\alpha)(\alpha\oplus\beta-x)
\quad\text{and}\quad
(x-\alpha)(\beta-x),
\]
and the former is no larger; the other child is analogous.  Refining every current leaf edge $m$ times is called the full depth-$m$ refinement.

\begin{lem}\label{lem:uniform-collapse}
Every leaf edge in the full depth-$m$ refinement of the root edge $e(k,k+1)$ has Euclidean diameter at most $(m+1)^{-1}$ and height at most $(2m+2)^{-1}$.  Consequently, the corresponding boundary height functions decrease uniformly to zero on $[k,k+1]$; equivalently, the boundary graphs converge to that interval in Hausdorff distance.
\end{lem}

\begin{proof}
At depth zero the endpoint denominators are $(1,1)$.  Passing from a pair $(q,s)$ to a child replaces it by $(q,q+s)$ or $(q+s,s)$.  Along each branch, the sum of the denominators therefore increases by at least one at every step.  At depth $m$,
\[
q+s\geq m+2.
\]
Since $(q-1)(s-1)\geq0$, we have $qs\geq q+s-1\geq m+1$.  The conclusion follows from \eqref{eq:farey-size}.
\end{proof}

We next install these finite refinements in separated bays.  The initial proper arc chain is
\[
\gamma_0=\bigcup_{k\in\ZZ}e(k,k+1).
\]
It is the graph of a continuous function $h_0\colon\RR\to[0,1/2]$.  Define
\[
\Omega_0=\{x+iy:y>h_0(x)\}.
\]
For $n\geq1$, choose the pairwise disjoint bays and their inner intervals
\[
B_n=[6n,6n+4],
\qquad
J_n=[6n+1,6n+3]\Subset\operatorname{int}B_n.
\]
For fixed $n$ and $m\geq0$, refine to depth $m$ each of the four root edges over $B_n$, leaving every other root edge unchanged.  Denote the resulting boundary function by $h_{n,m}$ and put
\[
D_{n,m}=\{x+iy:y>h_{n,m}(x)\}.
\]
On $B_n$, Lemma~\ref{lem:uniform-collapse} gives
\[
0\leq h_{n,m}(x)-h_{n,\infty}(x)\leq\frac1{2m+2},
\]
while the two functions agree off $B_n$.  The child arcs lie below their parents, so the boundary heights decrease with $m$.  Hence $\bigcup_m D_{n,m}=D_{n,\infty}$, and
\begin{equation}\label{eq:one-bay-exhaustion}
D_{n,m}\uparrow D_{n,\infty},
\end{equation}
where
\[
h_{n,\infty}(x)=
\begin{cases}
0,&x\in B_n,\\
h_0(x),&x\notin B_n,
\end{cases}
\qquad
D_{n,\infty}=\{x+iy:y>h_{n,\infty}(x)\}.
\]
All these domains contain $i$ and lie in $\HH$.

\section{The Green-energy diagonal construction}\label{sec:diagonal}

The next lemma converts the limiting divergence in one bay into a finite refinement.  Its proof is the point at which boundary regularity of the Green function enters.

\begin{lem}\label{lem:finite-refinement}
For every $n\geq1$ there exist $\varepsilon_n>0$ and an integer $m_n\geq0$ such that, with
\[
Q_n=\{x+iy:x\in J_n,\ 0<y<\varepsilon_n\},
\]
one has
\[
\int_{Q_n\cap D_{n,m_n}}
G_{D_{n,m_n}}(i,z)\bigl(\lambda^{\#}(z)\bigr)^2\dd A(z)>1.
\]
\end{lem}

\begin{proof}
The limiting domain $D_{n,\infty}$ agrees locally with $\HH$ along the interior of $B_n$.  The function
\[
z\longmapsto G_{D_{n,\infty}}(i,z)
\]
is positive and harmonic near $J_n$ in the upper half-plane and has no pole there.  Every point of $J_n$ is a regular boundary point, because the domain agrees locally with a half-disk.  The Green function therefore extends continuously by zero to $J_n$.
Its odd reflection across the real segment is harmonic, hence smooth, in a
full neighborhood of $J_n$.  The Hopf boundary point lemma
\cite[Lemma~3.4, p.~34]{GT01} gives
\[
\partial_y G_{D_{n,\infty}}(i,x)>0,
\qquad x\in J_n.
\]
By continuity and compactness, there exist constants
$c_n,\varepsilon_n>0$ such that
\[
\partial_y G_{D_{n,\infty}}(i,x+it)\geq c_n,
\qquad x\in J_n,\quad 0\leq t\leq\varepsilon_n.
\]
Since $G_{D_{n,\infty}}(i,x)=0$ for $x\in J_n$, integration in the
vertical direction yields
\begin{equation}\label{eq:green-linear-lower}
	G_{D_{n,\infty}}(i,x+iy)\geq c_n y,
	\qquad x\in J_n,\quad0<y<\varepsilon_n.
\end{equation}
By \eqref{eq:height-area} and \eqref{eq:green-linear-lower},
\[
\int_{Q_n}G_{D_{n,\infty}}(i,z)
\bigl(\lambda^{\#}(z)\bigr)^2\dd A(z)=\infty.
\]

Every point of $Q_n$ belongs to $D_{n,\infty}$, because the limiting boundary vanishes on $B_n$.  Extend $G_{D_{n,m}}(i,\cdot)$ by zero on $Q_n\setminus D_{n,m}$.  The domain exhaustion \eqref{eq:one-bay-exhaustion}, the Green-kernel convergence \eqref{eq:green-exhaustion}, and the uniform collapse of the boundary imply pointwise monotone convergence on all of $Q_n$ to $G_{D_{n,\infty}}(i,\cdot)$.  Monotone convergence now shows that the finite-depth integrals tend to infinity.  Some finite $m=m_n$ gives the required inequality.
\end{proof}

We choose the integer $m_n$ supplied by Lemma~\ref{lem:finite-refinement} in the $n$th bay and make no other change to $\gamma_0$.  More precisely, set
\[
h(x)=
\begin{cases}
h_{n,m_n}(x),&x\in B_n\text{ for some }n,\\
h_0(x),&x\notin\bigcup_{n\geq1}B_n,
\end{cases}
\]
and define
\[
\gamma=\{x+ih(x):x\in\RR\},
\qquad
\Omega=\{x+iy:y>h(x)\}.
\]
At each endpoint of a bay, both adjacent boundary arcs have height zero; hence $h$ is continuous.  Each $m_n$ is finite, and every compact real interval meets only finitely many bays, so $\gamma$ is a locally finite chain of Farey edges.  Moreover, $|x+ih(x)|\geq|x|$ and $0\leq h\leq1/2$.  Thus $\gamma$ is a continuous proper graph, $\Omega\subset\HH$, and $i\in\Omega$.

\begin{prop}\label{prop:infinite-energy}
The final domain satisfies
\begin{equation}\label{eq:omega-infinite-energy}
\int_\Omega G_\Omega(i,z)\bigl(\lambda^{\#}(z)\bigr)^2\dd A(z)=\infty.
\end{equation}
\end{prop}

\begin{proof}
Fix $n$.  On $B_n$ the final boundary agrees with $h_{n,m_n}$.  Outside $B_n$, every other bay refinement only lowers $h_0$, whereas $h_{n,m_n}=h_0$ there.  Thus
\[
h\leq h_{n,m_n}\quad\text{on }\RR,
\qquad
D_{n,m_n}\subset\Omega.
\]
Domain monotonicity gives
\[
G_\Omega(i,z)\geq G_{D_{n,m_n}}(i,z),
\qquad z\in D_{n,m_n}.
\]
The rectangles $Q_n$ are pairwise disjoint.  Hence, for $N\geq1$, Lemma~\ref{lem:finite-refinement} yields
\[
\begin{aligned}
\int_\Omega G_\Omega(i,z)\bigl(\lambda^{\#}(z)\bigr)^2\dd A(z)
&\geq\sum_{n=1}^N\int_{Q_n\cap D_{n,m_n}}
G_\Omega(i,z)\bigl(\lambda^{\#}(z)\bigr)^2\dd A(z)\\
&\geq\sum_{n=1}^N\int_{Q_n\cap D_{n,m_n}}
G_{D_{n,m_n}}(i,z)\bigl(\lambda^{\#}(z)\bigr)^2\dd A(z)>N.
\end{aligned}
\]
Letting $N\to\infty$ proves \eqref{eq:omega-infinite-energy}.
\end{proof}

\section{Reflection across the Farey boundary}\label{sec:reflection}

It remains to show that uniformization followed by Schwarz reflection produces a function meromorphic at every reflected vertex.  We begin with the open edges.

\begin{lem}\label{lem:real-edges}
The function $\lambda$ takes finite real values on the interior of every Farey edge.
\end{lem}

\begin{proof}
The positive imaginary axis is the hyperbolic geodesic joining $0$ to $\infty$, and $\lambda(iy)$ is real for $y>0$.  Every Farey edge is the image of this geodesic under an element of $\PSL_2(\ZZ)$.  Each transformation in \eqref{eq:anharmonic} has real coefficients and preserves $\widehat\RR$.  Thus $\lambda$ is real on every open Farey edge.  It is finite there because it is holomorphic and omits $\infty$ in $\HH$.
\end{proof}

Let $V$ be the vertex set of $\gamma$.  Every vertex is rational, $V$ is locally finite in $\CC$, and its only accumulation point in the Riemann sphere is $\infty$.  The local geometry at a vertex becomes especially transparent after a modular transformation.

\begin{lem}\label{lem:cusp}
Let $v\in V$, and let $e(v,u_1)$ and $e(v,u_2)$ be the two boundary edges incident to $v$.  There exist $A\in\PSL_2(\ZZ)$ with $A(v)=\infty$ and two distinct integers $k_1,k_2$ such that $A(e(v,u_j))$ is the vertical geodesic $\{\operatorname{Re}\zeta=k_j,\ \operatorname{Im}\zeta>0\}$.  Moreover, above some height, the end of $A(\Omega)$ corresponding to $v$ is exactly the finite-width strip between these two geodesics.  In particular,
\begin{equation}\label{eq:cusp-limit}
\operatorname{Im}A(w)\longrightarrow+\infty
\qquad\text{as }w\to v\text{ within }\Omega.
\end{equation}
\end{lem}

\begin{proof}
Write $v=p/q$ in reduced form, with $q>0$, and choose $a,b\in\ZZ$ so that $ap+bq=-1$.  Then
\[
A(w)=\frac{aw+b}{qw-p}
\]
belongs to $\PSL_2(\ZZ)$ and sends $v$ to $\infty$.  The modular group preserves Farey adjacency.  A Farey neighbor of $\infty=1/0$ has denominator one and is therefore an integer.  Hence $A(u_1)$ and $A(u_2)$ are distinct integers, and the incident geodesics become the corresponding vertical lines.

We determine the relevant side.  Label the adjacent vertices so that $u_1<v<u_2$.  If $u=r/s$ is either neighbor, direct calculation gives
\[
A(u)-\frac aq=-\frac{s}{q(qr-ps)}.
\]
The determinant $qr-ps$ equals $-1$ for the left neighbor and $1$ for the right neighbor, so the two integers $A(u_1)$ and $A(u_2)$ lie on opposite sides of $a/q$.  Furthermore,
\[
A(v+it)=\frac aq+\frac{i}{q^2t},
\qquad t>0,
\]
and $v+it\in\Omega$.

Local finiteness and the graph property provide a disk $U$ about $v$ meeting $\gamma$ only in the two incident edge germs.  The component of $\Omega\cap U$ accumulating at $v$ is the region between those germs.  Its image under $A$ therefore lies between the two vertical geodesics.  If $w\to v$ in $\Omega$, then $|A(w)|\to\infty$, while the real part remains between the two corresponding integers.  This proves \eqref{eq:cusp-limit}.

Finally, the sets $A^{-1}\{\operatorname{Im}\zeta>Y\}$ are horodisks shrinking to $v$ as $Y\to\infty$.  Choose $Y$ so large that such a horodisk lies in $U$.  Above height $Y$, the only boundary pieces of $A(\Omega)$ are the two vertical geodesics.  Since the middle interval contains $A(v+it)$ for small $t$, crossing either boundary geodesic exits the Jordan domain.  Thus, with $k_-=\min(k_1,k_2)$ and $k_+=\max(k_1,k_2)$,
\[
A(\Omega)\cap\{\operatorname{Im}\zeta>Y\}
=\{\zeta:k_-<\operatorname{Re}\zeta<k_+,\ \operatorname{Im}\zeta>Y\}.
\]
\end{proof}

The cusp description makes every exceptional reflected point an isolated meromorphic singularity.

\begin{prop}\label{prop:reflection}
There is a conformal map $W\colon\HH\to\Omega$ satisfying
\[
W(i)=i,
\qquad
W(\infty)=\infty
\]
in the sense of prime ends.  For such a map, $f=\lambda\circ W$ extends by reflection to a meromorphic function $F$ on $\CC$, and
\[
F^{-1}\{0,1,\infty\}\subset\RR.
\]
\end{prop}

\begin{proof}
Since $\gamma$ is a continuous proper graph, its one-point compactification $\gamma\cup\{\infty\}$ is a Jordan curve in the Riemann sphere.  Thus $\Omega$ is a Jordan domain on the sphere, with $\infty$ as a boundary point.  After applying a M\"obius transformation to reduce to bounded Jordan domains, the Riemann mapping theorem and Carath\'eodory's boundary theorem \cite[Chapter~6, Sections~1.1--1.2, especially p.~232]{Ahl79} give a conformal map from $\HH$ onto $\Omega$ that extends homeomorphically to the compactified boundaries.  Precomposition with an automorphism of $\HH$ first arranges that $\infty$ corresponds to $\infty$.  If $z_0=x_0+iy_0$ is then the preimage of $i$, further precomposition with $z\mapsto x_0+y_0z$ fixes $\infty$ and sends $i$ to $z_0$.  This gives the stated normalization.

Let
\[
E=W^{-1}(V)\subset\RR,
\]
where the boundary extension is understood.  Since $V$ has no finite accumulation point and $W(\infty)=\infty$, the set $E$ is locally finite in $\RR$.  On every component of $\RR\setminus E$, the boundary values of $W$ lie in the interior of one Farey edge.  Lemma~\ref{lem:real-edges} therefore gives finite real boundary values for $f$ there.  Define
\[
F(z)=
\begin{cases}
f(z),&\operatorname{Im}z>0,\\
\overline{f(\overline z)},&\operatorname{Im}z<0,
\end{cases}
\]
and fill in $\RR\setminus E$ by continuous extension.  Schwarz reflection \cite[p.~172, Theorem~24]{Ahl79} shows that $F$ is holomorphic on $\CC\setminus E$.

Fix $x_v\in E$, with $W(x_v)=v\in V$, and choose a disk about $x_v$ containing no other point of $E$.  Reflection across the two punctured real intervals makes $F$ holomorphic on the full punctured disk.  Let $A$ be as in Lemma~\ref{lem:cusp}.  Boundary continuity and \eqref{eq:cusp-limit} imply
\[
\operatorname{Im}A(W(z))\longrightarrow+\infty
\qquad(z\to x_v,\ z\in\HH).
\]
Apply \eqref{eq:anharmonic} to $A^{-1}$ and put $\zeta=A(W(z))$.  Then
\[
f(z)=\lambda(W(z))
=R_{A^{-1}}\bigl(\lambda(A(W(z)))\bigr).
\]
Equation~\eqref{eq:lambda-cusp} yields, in the spherical metric,
\[
f(z)\longrightarrow a_v:=R_{A^{-1}}(0)\in\{0,1,\infty\}.
\]
The same spherical limit holds in the reflected lower half-disk.  If $a_v=0$ or $1$, the singularity is removable.  If $a_v=\infty$, then $|F(z)|\to\infty$ as $z\to x_v$ through the punctured disk.  Thus $1/F$ is holomorphic there, tends to zero, and has a removable singularity.  Since $1/F$ is not identically zero, its zero has finite order, and $F$ has a pole of finite order at $x_v$.

Since $E$ is locally finite, filling in all its points produces a meromorphic function on $\CC$.  It is nonconstant because $\lambda\circ W$ is nonconstant in $\HH$.  Both open half-planes omit $0$, $1$, and $\infty$, so every preimage of these values is real.
\end{proof}

The normalization at the boundary point $\infty$ is not cosmetic.  We record its role explicitly.

\begin{rem}\label{rem:infinity-normalization}
The condition $W(\infty)=\infty$ ensures that the preimages of vertices escaping to infinity cannot accumulate at a finite point of $\RR$.  Without it, the reflected function need not be meromorphic at such an accumulation point.
\end{rem}

\section{Proof of the main theorem}\label{sec:main-proof}

We now combine the reflection theorem with the Green-energy obstruction.

\begin{proof}[Proof of Theorem~\ref{thm:main}]
Let $\Omega$ be the domain constructed in Sections~4--5, let $W$ be the map in Proposition~\ref{prop:reflection}, and set
\[
f=\lambda\circ W\qquad\text{in }\HH.
\]
Proposition~\ref{prop:reflection} supplies a nonconstant meromorphic extension $F$ to $\CC$ and shows that $f$ omits $0$, $1$, and $\infty$ in $\HH$.

Green functions and spherical area are conformally invariant.  Since $W(i)=i$, the change of variables $w=W(z)$ gives
\[
\int_\HH G_\HH(i,z)\bigl(f^{\#}(z)\bigr)^2\dd A(z)
=\int_\Omega G_\Omega(i,w)\bigl(\lambda^{\#}(w)\bigr)^2\dd A(w)
=\infty
\]
by \eqref{eq:omega-infinite-energy}.  Proposition~\ref{prop:energy-obstruction} implies $f\notin N(\HH)$, completing the proof.
\end{proof}

\section{Further remarks}\label{sec:further-remarks}

We use the standard Nevanlinna notation
\[
\rho(F)=\limsup_{r\to\infty}\frac{\log T(r,F)}{\log r},
\qquad
\delta(a,F)=1-\limsup_{r\to\infty}\frac{N(r,a;F)}{T(r,F)}.
\]

Theorem~\ref{thm:main} does not conflict with Nevanlinna's classical positive theorem.  Nevanlinna \cite{Nev25} proved that a finite-order meromorphic function on $\CC$ which omits three values in a half-plane is of bounded type there.  See also the modern discussion in \cite[Section~2.1]{EKS26}.  Consequently, the function constructed above has infinite order.

Nor does Edrei's theorem on radially distributed values rule out the example.  Specialized to derivative order $l=0$ and to the two rays with arguments $0$ and $\pi$, \cite[Theorem~1]{Edr55} says that if all zeros, one-points, and poles of a meromorphic function lie on these rays and
\begin{equation}\label{eq:deficiency-positive}
\delta(0,F)+\delta(1,F)+\delta(\infty,F)>0,
\end{equation}
then $F$ has finite order.  More precisely, both complementary angular gaps are $\pi$, so Edrei's bound is
\[
\rho(F)\leq\max\left\{\frac\pi\pi,\frac\pi\pi\right\}=1.
\]
The positive-deficiency hypothesis is part of the original theorem and is reproduced explicitly in \cite[Theorem~A]{WZ14}.  

If \eqref{eq:deficiency-positive} held for our function, Edrei's theorem would make its order finite, and Nevanlinna's theorem would then put $F|_\HH$ in $N(\HH)$, a contradiction.  Since deficiencies are nonnegative, the construction necessarily satisfies
\[
\delta(0,F)=\delta(1,F)=\delta(\infty,F)=0.
\]

Finally, we have neither used nor proved that the complete set of singular values of the reflected function is exactly $\{0,1,\infty\}$.  That stronger assertion would require a separate analysis of possible asymptotic values at infinity.

%\section*{Declaration of competing interest}
%The author declares no competing interests.
%
%\section*{Data availability}
%No data was used for the research described in this article.
%
%\section*{Acknowledgments}
%Teng Zhang is supported by the China Scholarship Council, the Young Elite Scientists Sponsorship Program for PhD Students (China Association for Science and Technology), and the Fundamental Research Funds for the Central Universities at Xi'an Jiaotong University (Grant No.~xzy022024045).

\end{document}